\documentclass[12pt,reqno] {amsart}
\usepackage{amsfonts}
\usepackage{amsmath}
\usepackage{amssymb}
\usepackage{mathrsfs}
\usepackage{amsthm}
\usepackage{fancyhdr}
\usepackage{graphicx}
\usepackage{extarrows}
\usepackage{indentfirst,latexsym,bm}
\usepackage{amsfonts,amscd}
\usepackage{graphicx}
\usepackage[all]{xy}
\usepackage{pst-node}
\usepackage{pstricks}
\usepackage{cite}
\usepackage{enumitem}
\usepackage{subfigure}
\usepackage{pifont}
\usepackage{amsbsy}
\usepackage{bbm}
\usepackage{ulem}
\usepackage{tikz}
\usepackage{float}

\newtheorem{theorem}{Theorem}[section]
\newtheorem*{theorem*}{Theorem}
\newtheorem{lemma}[theorem]{Lemma}
\newtheorem*{lemma*}{Lemma}

\newtheorem*{proposition*}{Proposition}
\newtheorem{corollary}[theorem]{Corollary}

\newtheorem*{theorem0}{Theorem 0}

\makeatletter
\renewcommand{\@thesubfigure}{\hskip\subfiglabelskip}
\makeatother

 \theoremstyle{plain}

\def\bC{\mathbb{C}}
\def\bD{\mathbb{D}}

\def\bN{\mathbb{N}}

\def\bQ{\mathbb{Q}}
\def\bR{\mathbb{R}}

\def\cB{\mathcal{B}}

\def\cE{\mathcal{E}}
\def\cF{\mathcal{F}}

\def\cH{\mathcal{H}}

\def\cP{\mathcal{P}}

\def\fR{\mathfrak{R}}

\def\Re{\text{Re}}

\def\i{\mathfrak{i}}

\def\<{\langle}
\def\>{\rangle}

\def\gcd{\textsf{gcd}}

\begin{document}

\begin{center}
\large On a Hilbert Space Reformulation of Riemann Hypothesis
\end{center}
\bigskip
%\begin{center}
%Ge's group
%\end{center}

\begin{center}
Boqing Xue\footnote{Institute of Mathematical Sciences, ShanghaiTech University, Shanghai, China, 201210. Email: xuebq@shanghaitech.edu.cn}
\end{center}

\bigskip

\begin{quote}
{\footnotesize {
\begin{center}
\scshape Abstract
\end{center}
}
 \medskip

We explore Hilbert space reformulations of Riemann Hypothesis developed by Nyman, Beurling, B\'{a}ez-Duarte, et. al. with a weighted Bergman space $\cH=A_1^2(\bD)$, i.e., Riemann hypothesis holds if and only if the Hilbert subspace $\cH_0$ spanned by a certain family of functions coincides with $\cH$. A condition that a function does not belong to $\cH_0^\bot$ is given. Moreover, it is proved that the von-Neumann algebra generated by a certain monoid $T_\bN=\{T_k:\, k\in \bN\}$ of operators is exactly $B(\cH)$. As a result, Riemann hypothesis is true if and only if $\cH_0$ is $T_k^\ast$-invariant for all $k\in \bN$.
}
\end{quote}

%11N99 20F05 20F38 20F65

%\tableofcontents

\bigskip

\section{Introduction}

A key of understanding the natural numbers $\bN=\{1,2,3,\ldots\}$ is to study the primes $\cP=\{2,3,5,7,11,\ldots\}$. Riemann \cite{Riemann} studied the multiplicative structure of natural numbers by a complex-valued function $\zeta(s)=\sum\nolimits_{n=1}^\infty \frac{1}{n^s}$ $(\Re(s)>1)$, which can be extended to a meromorphic function on the whole complex plane. In classical statistical physics, the structure of a system is usually described by a partition function, whose logarithmic derivative tells the average of energy. Correspondingly, the logarithm derivative of Riemann $\zeta$-function, i.e.,
\[
\frac{d}{ds}\log \zeta(s) = \frac{\zeta^\prime(s)}{\zeta(s)}=-\sum\limits_{n=1}^\infty\frac{\Lambda(n)}{n^s},
\]
tells us the average information of the number of primes. Here $\Lambda$ is called von Mangoldt function and can be viewed as a weighted characteristic function of the primes.

More concretely, let us denote by $\rho$ the non-trivial zeros of $\zeta(s)$. Also let $\Lambda(x)=0$ when $x\notin \bN$. Then the sum of weighted number of primes up to $x$ is shown by the Riemann-von Mangoldt explicit formula
\[
\sum\limits_{n\leq x}\Lambda(n)+\frac{1}{2}\Lambda(x)=
x-\sum\limits_{\rho}\frac{x^\rho}{\rho} -\frac{1}{2}\log
(1-x^{-2})-\log 2\pi.
\]
Note that the modulus of $x^\rho$ is bounded by $x^{\Re\rho}$. After some efforts, one can deduce that
\[
\sum\limits_{n\leq x}\Lambda(n) = x + \textit{O}\left(x^{\delta+\varepsilon}\right)
\]
for any $\varepsilon>0$, where $\delta=\sup\nolimits_{\rho}\Re(\rho)$. Riemann hypothesis (see \cite{Bom06,GX18} for surveys), which is still open, says that all the non-trivial zeros of $\zeta(s)$ lie on the vertical line with real part $1/2$. Or equivalently, the prime distribution should satisfy $\sum\nolimits_{n\leq x}\Lambda(n) = x+\textit{O}(x^{1/2+\varepsilon})$ for any $\varepsilon>0$.

In 1896, Hadamard \cite{Hadamard} and de la Vall\'{e}e-Poussin \cite{VPoussin} proved independently that $\zeta(s)$ is non-zero for $\Re(s)\geq 1$, which implies the prime number theorem $\sum\nolimits_{n\leq x}\Lambda(n) \sim x$. Till now, people only have knowledge of the non-existence of zeros in the region that is ``very close'' to the left of the vertical line $\Re(s)=1$ (see \cite{Vin} for example).

In quantum physics, observables are described by operators on Hilbert spaces instead of functions. Inspired by physical points of views, Hilbert-P\'{o}lya conjecture is regarded as a promising approach for studying Riemann hypothesis. This conjecture says that the imaginary parts of the zeros of $\zeta(s)$ are eigenvalues of a self-adjoint operator $D$. Since the spectrum of a self-adjoint (possibly unbounded and densely defined) operator on a Hilbert space is contained in $\bR$, then $\i(\rho-1/2)\in \bR$ implies that $\Re(\rho)=1/2$, where $\rho$ is any non-trivial zero of $\zeta(s)$.

Differential operators usually serve as infinitesimal operators associated to underlying structures. And they are always self-adjoint (after multiplying the imaginary unit $\i$ in some situations). So they are suitable candidates for Hilbert-P\'{o}lya conjecture. In \cite{Con99}, Connes showed a beautiful construction with a differential operator $D$ on $L^2_w(C_\bQ)$, where $C_\bQ$ is the id\'{e}l\`{e} class group of the rational numbers $\bQ$. For some technical reasons (to ensure that the eigenfunctions lie in the Hilbert space), Connes added a weight $w$ to the Haar measure of $C_\bQ$. Due to this weight, it is a pity that the differential operator $D$ is not self-adjoint anymore. However, Connes showed a trace formula, whose global validity is also equivalent to generalized Riemann hypothesis.

In a sequence of papers \cite{Ge01,Mey04,Li20,GLWX}, people studied the multiplicative differential operator $x \frac{d}{dx}$ defined on several Schwartz-type spaces on $(0,+\infty)$. In \cite{GLWX}, it is proved that the eigenvalues of $x \frac{d}{dx}$ coincides with the non-trivial zeros of $\zeta(s)$, and their corresponding multiplicities are same. However, it is a pity that the concept of self-adjoint operator on general topological vector spaces is not mature at present. Also see \cite{BK99-1,BK99-2,BK11} for some of the constructions related to Hilbert-P\'{o}lya conjecture. Moreover, a discrete-type `differential operator' induced by a convolution of the M\"{o}bius function was studied in \cite{DHX}, which is also related to the Riemann $\zeta$-function.

Another approach for studying Riemann hypothesis is to consider Hilbert space reformulations originated by Nyman \cite{Nym50} and developed by Beurling \cite{Beu55}, et. al. Each reformulation involves a Hilbert space $H$, a family $F$ of functions, and a function $\beta_0$ in $H$. Let $H_0$ be the Hilbert subspace spanned by $F$ in $H$. The Hilbert space reformulation of Riemann hypothesis is stated as the following:

\begin{center}
\textsl{Riemann hypothesis holds if and only if $H_0=H$, if and only if $\beta_0\in H_0$.}
\end{center}

For $0<\lambda\leq 1$, Nyman and Beurling considered the family of functions $F=\{\rho_\lambda:\, x\in (0,1]\}$ in the Hilbert space $H=L^2(0,1]$, where
\[
\rho_\lambda(x) =\left\{\frac{\lambda}{x}\right\} - \lambda \left\{\frac{1}{x}\right\}.
\]
Here $\{y\}$ means the fractional part of a real number $y$. Their choice of $\beta_0$ is the constant function $1$ on $(0,1]$. Moreover, Beurling \cite{Beu55} also proved that $\zeta(s)$ is free from zero in the right half-plane $\Re(s)>1/p$ ($1<p<\infty$) if and only if $\text{span}(F)$ is dense in $L^p(0,1]$.

In \cite{BD03}, Baez-Duarte showed a stronger version. The uncountable family $\{\rho_\lambda:\, x\in (0,1]\}$ was replaced by a countable family $F=\{\rho_{1/k}:\, k=2,3,4,\ldots\}$. Baez-Duarte's choice of $H$ consists of functions in $L^2(0,1]$ which are almost everywhere constant on each of the sub-intervals $\left(\frac{1}{n+1},\frac{1}{n}\right]$ $(n=1,2,3,\ldots)$. The function $\beta_0$ is still the constant function $1$.

Applying a unitary equivalence of Hilbert spaces, Bagchi \cite{Bag06} gave a discrete version of reformulation, with $H=l^2(\bN,\nu)$, $\beta_0=1$ and $H_0$ spanned by $F=\{r_k:\, k=2,3,\ldots\}$. Here $\nu$ is the probability measure on $\bN$ induced by $\nu(\{n\})=\frac{1}{n(n+1)}$, and $r_k(n) = \left\{\frac{n}{k}\right\}$ $(n\in \bN)$. In fact, it is not hard to prove that $\beta_0$ can be replaced by any function such that (letting $\beta_0(0)=0$)
\[
\sum\limits_{n=1}^\infty \frac{\beta_0(n)-\beta_0(n-1)}{n^s}
\]
can be defined and is non-zero in the vertical strip $1/2<\Re(s)<1$. To approximate $\beta_0$ by combinations of functions $r_k$ $(k=2,3,\ldots)$, the best possible coefficients are related to the M\"{o}bius function $\mu$. We refer to \cite{BS1,BS2,BS3,BS4,Wei07} for details.

Recently, Waleed Noor \cite{Noo19}  showed another reformulation with $H=H^2(\bD)$, the hardy space on the open unit disk $\bD=\{z\in \bC:\, |z|<1\}$. The Hilbert subspace $H_0$ is spanned by the family $F=\{h_k:\, k=2,3,\ldots\}$, and $\beta=1$. Here $h_k(z) = \frac{1}{1-z} \log \left(\frac{1}{k}(1+z+z^2+\ldots+z^{k-1})\right)$ $(z\in \bD)$.

Waleed Noor made use of a weighted Bergman space as a bridge between $l^2(\bN,\nu)$ and Hardy space. For a complex number $z$, we use both the expressions $z=x+\i y$ for some $x,y\in \bR$ and $z=re^{i\theta}$ for some $r\geq 0$, $0\leq \theta<2\pi$. The normalized area measure on $\bD$ is denoted by $dA$, i.e., $dA=\pi^{-1}dxdy=\pi^{-1}rdrd\theta$. And let $dA_1$ be the probability measure
\[
dA_1(z)=2(1-|z|^2)dA(z)=2\pi^{-1}r(1-r^2)drd\theta.
\]
on $\bD$. %(Comparing with the measure defined in \cite{HKZ}, we discard a factor $2$ for simplicity.)
Then $L^2(\bD,dA_1)$ forms a Hilbert space with norm
\[
\|f\|_2 = \left(\int_{\bD} |f(z)|^2 dA_1(z)\right)^{1/2}.
\]
The weighted Bergman space $A_1^2(\bD)$ consists of all the analytic functions in $L^2(\bD,dA_1)$. Write $\bN_0=\{0,1,2,\ldots\}$. Then $A_1^2(\bD)$ has an orthogonal basis $\{z^m:\, m\in \bN_0\}$. In particular, any function $f$ in $A_1^2(\bD)$ is a power series $f(z)=\sum\nolimits_{n=0}^\infty a_{n}z^n$ for $z\in \bD$. The norm of $f$ in $A_1^2(\bD)$ equals
\begin{equation} \label{eq_norm_Bergman}
\|f\|_2 = \left(\sum\limits_{n=0}^\infty \frac{2|a_{n}|^2}{(n+1)(n+2)}\right)^{1/2}.
\end{equation}
We refer to \cite{HKZ} for details.

Define $\Psi: \, l^2(\bN,\nu)\rightarrow A_1^2(\bD)$ by $(\Psi f)(z) = \frac{1}{\sqrt{2}}\sum\nolimits_{n=0}^\infty f(n+1)z^n$. It can be verified that $\Psi$ is an isometric isomorphism, and $\Psi r_k = \frac{1}{\sqrt{2}}s_k$ ($k=2,3,\ldots$), where
\[
s_k(z)= \frac{1}{1-z}\frac{d}{dz}\log \left(\sum\limits_{j=0}^{k-1}z^j\right) = \frac{1+2z+\ldots + (k-1)z^{k-2}}{1-z^k}.
\]
These functions seem interesting, since they look like weighted logarithmic derivatives of partition functions that carry information on natural numbers $k=2,3,4,\ldots$. Moreover, one also has $\Psi 1=\frac{1}{\sqrt{2}}\frac{1}{1-z}$. In the rest of this paper, we take $\cH=A_1^2(\bD)$, $\beta(z)=\frac{1}{1-z}$ and let $\cH_0$ be the Hilbert subspace spanned by $\cF=\{s_k:\, k=2,3,\ldots\}$. Now we restate a Hilbert space reformulation as below.

\begin{theorem0}
Riemann hypothesis holds if and only if $\cH_0=\cH$, if and only if $\beta\in \cH_0$.
\end{theorem0}

Now it is natural to ask the following questions. %(\uppercase\expandafter{\romannumeral1}) Which kind of functions belong to $\cH_0$?
(\uppercase\expandafter{\romannumeral1}) Which kind of functions does not belong to $\cH_0^\bot$? Here $\cH_0^\bot$ is the orthogonal complement of $\cH_0$ in $\cH$. (\uppercase\expandafter{\romannumeral2}) For which kind of families $\cE$ of functions, the subspace $\text{Span}(\cE)$ is dense in $\cH$? Note that $\cH_0=\cH$ is equivalent to $\cH_0^\bot=\{0\}$. For question (\uppercase\expandafter{\romannumeral1}), we hope to find as many functions as we can. For any acquired family $\cE$ in question (\uppercase\expandafter{\romannumeral2}), the remaining problem becomes that whether functions in $\cE$ can be approximated by finite combinations of the functions $s_k$ $(k=2,3,\ldots)$. %Theorem 9 and 12 of Noor are results towards question (\uppercase\expandafter{\romannumeral3}) and (\uppercase\expandafter{\romannumeral2}),respectively.
In this paper, we obtain a few partial results to these questions.

For $k=2,3,\ldots$, let
\[
f_k(z)=(1-z)s_k(z)=\frac{d}{dz}\log \left(\sum\limits_{j=0}^{k-1}z^j\right),\quad (z\in \bD).
\]
\begin{theorem}  \label{thm_spanfk=A_1^2}
The subspace $\text{Span}\{f_k:k=2,3,\ldots\}$ is dense in $\cH$.
\end{theorem}

\begin{theorem} \label{thm_notin_Hbot}
If $g$ is a non-zero function in $\cH$ such that $\frac{g(z)}{1-z}\in \cH$, then $g\notin \cH_0^\bot$.
\end{theorem}

The proof of the Hilbert reformulation of Riemann hypothesis of Nyman and Beurling involved a monoid of operators $\{T_\lambda:\, \lambda\in (0,1]\}$ (see \cite{Beu55}). Waleed Noor \cite{Noo19} also showed that the Riemann hypothesis is true if and only if $\cH_0$ contains a cyclic vector for a certain monoid of operators $\{W_n:\, n\in \bN\}$. Instead of considering a single operator (such as the differential operators for Hilbert-P\'{o}lya conjecture), the structure of the multiplicative monoid $(\bN,\cdot)$ can also be studied by monoids of operators.

In \cite{DHX}, Dong, Huang and the author considered the left regular action of $\bN$ on $l^2(\bN)$. More concretely, for $k\in \bN$, let $L_k$ be the operator defined by
\[
(L_k f)(n) =
\begin{cases}
f(n/k),\quad &\text{if }k|n,\\
0,\quad &\text{otherwise}.
\end{cases}
\]
Then the multiplicative monoid of operators $L_\bN=\{L_k:\, k\in \bN\}$ shares the same structure with $\bN$. In particular, certain operator algebras generated by $L_\bN$ can reflect properties of natural numbers. The following results are shown in \cite{DHX}:
(\romannumeral1) The maximal ideal space of the Banach algebra generated by $L_\bN$ in $B(l^2(\bN))$ is homeomorphic to $\bD^\cP$.
(\romannumeral2) The $C^\ast$-algebra generated by $L_\bN$ in $B(l^2(\bN))$ contains no projections of finite rank other than $0$.
(\romannumeral3) The von Neumann algebra generated by $L_\bN$ in $B(l^2(\bN))$ is exactly $B(l^2(\bN))$.
In particular, statement (\romannumeral1) reflects the fundamental theorem of arithmetic. And statement (\romannumeral2) is equivalent to the existence of infinitely many primes.

Now let us turn back to the Hilbert spaces $\cH$ and $\cH_0$. Note that $\cH_0=\cH$ if and only if $\cH_0$ is an invariant subspace of all operators in $B(\cH)$. This fact inspires us to consider the following approach for Riemann Hypothesis: first, to find a family $\{T_\omega\}_{\omega\in \Omega}$ of bounded operators on $\cH$ such that $\cH_0$ is $T_w$-invariant for all $\omega\in \Omega$; second, to show that the weak-operator closed algebra generated by $\{T_\omega\}_{\omega\in \Omega}$ in $B(\cH)$ is exactly $B(\cH)$. %In this paper, we show some partial results to this approach.

For $k\in \bN$, define the operators $T_k$ on $\cH$ by
\begin{equation} \label{eq_Tk_def}
(T_k f)(z) = k^{1/2}z^{k-1}f(z^k)\frac{1-z^k}{1-z},\quad (z\in \bD,\, f\in \cH).
\end{equation}
%Note that $\cH=A_1^2(\bD)$ has an orthonormal basis $\{z^m:\, m=0,1,2,\ldots\}$. For $k\in \bN$, define the operators $T_k$ on $\cH$ by
%\[
%T_k (z^m) = k^{1/2}z^{km+k-1}\sum\limits_{l=0}^{k-1} z^{l} = \frac{k^{1/2}z^{km+k-1}(1-z^k)}{1-z}, \quad (m\in \bN_0, \, z\in \bD).
%\]
Then $T_1=I$ is the identity operator on $\cH$. %For $j,k\in \bN$, we have
%\begin{align*}
%(T_jT_k) (z^m) &= k^{1/2} \sum\limits_{l=0}^{k-1} T_j(z^{km+k-1+l}) = (jk)^{1/2}  \sum\limits_{l=0}^{k-1}\sum\limits_{n=0}^{j-1}z^{j(km+k-1+l)+j-1+n} \\
%&= (jk)^{1/2}\sum\limits_{l=0}^{k-1}\sum\limits_{n=0}^{j-1}z^{jkm+jk-1+(jl+n)} = T_{jk}(z^m), \quad (m\in \bN_0, \, z\in \bD).
%\end{align*}
%So $T_jT_k=T_{jk}$ for any $j,k\in \bN$. Let $f(z)=\sum\nolimits_{n=0}^\infty a_nz^n$ be a function in $\cH$. Then
%\[
%(T_k f)(z) = k^{1/2}\sum\limits_{n=0}^\infty a_n z^{kn+k-1}\sum\limits_{l=0}^{k-1}z^l = k^{1/2}\sum\limits_{m=k-1}^\infty a_{\kappa(k;m)}z^m,
%\]
%where $\kappa(k;m)=\lfloor (m+1)/k\rfloor-1 $. So
%\begin{align*}
%\|T_kf\|^2 &= k\sum\limits_{m=k-1}^\infty \frac{2|a_{\kappa(k;m)}|^2}{(m+1)(m+2)} = k\sum\limits_{n=0}^\infty 2|a_n|^2 \sum\limits_{l=kn+k-1}^{kn+2k-2}\frac{1}{(l+1)(l+2)}\\
%&= k\sum\limits_{n=0}^\infty 2|a_n|^2 \cdot \left(\frac{1}{kn+k}-\frac{1}{kn+2k}\right) = \sum\limits_{n=0}^\infty \frac{2|a_n|^2}{(n+1)(n+2)}  = \|f\|^2.
%\end{align*}
%It follows that $\|T_k\|=1$.
It can be verified that $T_\bN:=\{T_k:\, k\in \bN\}$ forms an abelian monoid of bounded operators on $\cH$ (see section \ref{sec_3}). Similar as in \cite{Beu55} and \cite{Noo19}, these operators satisfy the following property.

\begin{lemma} \label{lem_Tk_sm}
For any $k,m\in \bN$ with $m\geq 2$, 
\[
T_k s_m = k^{1/2}\Big(s_{km}-\frac{1}{m}s_k\Big).
\]
In particular, we have $T_k(\cH_0) \subseteq \cH_0$.
\end{lemma}

The family $T_\bN$ seems not enough to generate $B(\cH)$ as a weak-operator closed algebra. So we put their adjoints $T_k^\ast$ together. A von-Neumann algebra is a unital $\ast$-algebra of bounded operators on a Hilbert space that is closed in the weak operator topology. Inspired by statement (\romannumeral3) which is shown as previous, we prove the following theorem.

\begin{theorem} \label{thm_TN_generate_BH}
The von-Neumann algebra generated by $T_\bN$ in $B(\cH)$ is exactly $B(\cH)$.
\end{theorem}

As a corollary, we conclude the following.

\begin{corollary} \label{cor_RH_Lkast_invariant}
Riemann hypothesis is true if and only if $\cH_0$ is $T_k^\ast$-invariant for all $k\in \bN$.
\end{corollary}

By Corollary \ref{cor_RH_Lkast_invariant}, it would be interesting to investigate the invariant subspaces of $L_k^\ast$ $(k\in \bN)$. A short discussion is held at the end of Section \ref{sec_3}. In particular,  the function $\beta(z)=\frac{1}{1-z}$ is a common eigenfunction of all the operators $T_k^\ast$ $(k\in \bN)$. We ask the following questions.

Question 1: Can we determine the spectrum of $T_k$ and $T_k^\ast$ for a given $k$ with $k\geq 2$?

Question 2: Is $\beta$ the only common eigenfunction of all the operators $T_k^\ast$ $(k\in \bN)$?

Question 3: Are there any non-trivial infinite-dimensional invariant subspaces of $\{T_k^\ast:\, k\in \bN\}$? What properties do they have?

Question 4: Are there any non-trivial finite-dimensional invariant subspaces of $T_\bN$?

These invariant spaces might be related to Kadison transitive algebra problem or (hyper)invariant subspace problem.

This paper is organized as follows. In Section \ref{sec_2}, we prove Theorems \ref{thm_spanfk=A_1^2} and \ref{thm_notin_Hbot}. In Section \ref{sec_3}, we study the properties of $T_k$ and $T_k^\ast$, and prove Lemma \ref{lem_Tk_sm}, Theorem \ref{thm_TN_generate_BH} and Corollary \ref{cor_RH_Lkast_invariant}. For a complex-valued function $f$ and a non-negative valued function $g$, the notation $f\ll g$ means that $|f|\leq c g$ for some absolute constant $c>0$. For basics in number theory, we refer to \cite{Iwa} and \cite{Nat}. For basics in Bergman spaces, see \cite{HKZ}. For those in operator algebra and functional analysis, see \cite{Kad-Rin}.

\section{On Hilbert Space Reformulation}
\label{sec_2}

The Hardy space $H^2(\bD)$ consists of analytic functions on $\bD$ such that
\[
\|f\|_{H^2} = \sup\limits_{0\leq r<1} \left(\frac{1}{2\pi}\int_0^{2\pi}|f(re^{\i \theta})|^2d\theta\right)^{1/2}<+\infty.
\]
When $f(z)=\sum\nolimits_{n=0}^\infty a_n z^n$, one also has
\begin{equation} \label{eq_norm_H2}
\|f\|_{H^2} = \sum\limits_{n=0}^\infty |a_n|^2.
\end{equation}
Combining \eqref{eq_norm_Bergman} and \eqref{eq_norm_H2}, it is not hard to obtain the following lemma.

\begin{lemma} \label{lem_equiv_H^2_A_1^2}
Let $G_l,G\in H^2(\bD)$ $(l=1,2,\ldots)$. Suppose that $\|G_l-G\|_{H^2}\rightarrow 0$ as $l\rightarrow \infty$. Denote $g_l=G_l^\prime$ and $g=G^\prime$. Then $\|g_l-g\|_2\rightarrow 0$ as $l\rightarrow \infty$.
\end{lemma}

\begin{proof}
Without loss of generality, we can assume that $G=0$. (Otherwise we substitute $G_l-G$ for $G_l$.) Let $G_l(z)=\sum\limits_{n=0}^\infty a_n^{(l)}z^n$. Then
\[
\sum\limits_{n=0}^\infty |a_n^{(l)}|^2 = \|G_l\|_{H^2} \rightarrow 0,\quad (l\rightarrow \infty).
\]
Now $g_l(z)=G_l^\prime(z)=\sum\nolimits_{n=1}^\infty a_n^{(l)}n z^{n-1}$. It follows that
\[
\|g_l\|_2 = \sum\limits_{n=1}^\infty \frac{2|a_n^{(l)}n|^2}{n(n+1)}\leq 2\sum\limits_{n=0}^\infty |a_n^{(l)}|^2 \rightarrow 0,\quad (l\rightarrow \infty).
\]
The lemma then follows.
\end{proof}

Now we prove Theorem \ref{thm_spanfk=A_1^2}.

\begin{proof} [Proof of Theorem \ref{thm_spanfk=A_1^2}]
Let $F_k(z) = \log \frac{1-z^k}{1-z}$ $(z\in \bD)$ for $k\in \bN$. Then $F_k^\prime=f_k$. By Lemma \ref{lem_equiv_H^2_A_1^2}, it is sufficient to prove that, for any given $m\geq 1$,
\[
z^m \in \overline{\text{Span}}^{\|\cdot\|_{H^2}}\{F_k:k=2,3,\ldots\}.
\]
Note that
\[
F_k(z) = \sum\limits_{n=1}^\infty \frac{z^n}{n}-\sum\limits_{m=1}^\infty \frac{z^{km}}{m} = \sum\limits_{n=1}^\infty \left(\frac{1}{n}-\delta_{k|n}\cdot\frac{k}{n}\right) z^n.
\]
Here $\delta_{k|n}$ equals $1$ when $k|n$ and equals $0$ otherwise. Also let $\delta_{m|k|n}$ be $1$ when $m|k$ and $k|n$, and be $0$ otherwise. And let $\delta_{m=n}$ be $1$ when $m=n$ and be $0$ otherwise. %For convenience, we set $\mu(x)=0$ for $x\notin \bN$.
Let $K$ be any natural number with $K \geq m$. Note that
\begin{align*}
&\left\|\sum\limits_{k=1}^K \frac{\mu(k/m)\delta_{m|k}}{k/m} F_k(z) + z^m\right\|_{H^2}^2 = \sum\limits_{n=1}^\infty \left|\sum\limits_{k=1}^K \frac{\mu(k/m)\delta_{m|k}}{k/m}\left(\frac{1}{n}-\delta_{k|n}\cdot\frac{k}{n}\right)+\delta_{m=n}\right|^2\\
&\quad\quad\quad \ll \sum\limits_{n=1}^\infty \frac{1}{n^2}\left|\sum\limits_{k=1}^K \frac{\mu(k/m)\delta_{m|k}}{k/m}\right|^2+\sum\limits_{n=1}^\infty \left|\sum\limits_{k=1}^K \frac{\mu(k/m)\delta_{m|k}}{k/m}\left(-\delta_{k|n}\cdot\frac{k}{n}\right)+\delta_{m=n}\right|^2.
\end{align*}
By the equality
\begin{equation} \label{eq_mu}
\sum\limits_{k=1}^\infty \mu\left(\frac{k}{m}\right)\delta_{m|k|n} = \delta_{m=n},
\end{equation}
one deduces that, for $n \leq K$,
\[
\sum\limits_{k=1}^K \frac{\mu(k/m)\delta_{m|k}}{k/m}\cdot\left(-\delta_{k|n}\cdot\frac{k}{n}\right)+\delta_{m=n}= 0.
\]
Let $\tau$ be the divisor function, i.e., $\tau(n)$ counts the number of different divisors of $n$. It satisfies that $\tau(n)\ll n^\varepsilon$ for any $\varepsilon>0$. For $n>K$, we have
\[
\left|\sum\limits_{k=1}^K \frac{\mu(k/m)\delta_{m|k}}{k/m}\left(-\delta_{k|n}\cdot\frac{k}{n}\right)+\delta_{m=n} \right|\leq \frac{m}{n}\cdot \tau(n).
\]
It follows that
\begin{align*}
&\left\|\sum\limits_{k=1}^K \frac{\mu(k/m)\delta_{m|k}}{k/m} F_k(z) + z^m\right\|_{H^2}^2\\
&\quad \quad \quad\ll \left(\sum\limits_{n=1}^\infty \frac{1}{n^2}\right)\cdot \left|\sum\limits_{1\leq k^\prime\leq K/m} \frac{\mu(k^\prime)}{k^\prime}\right|^2+ \sum\limits_{n>K}\left|\frac{m\tau(n)}{n}\right|^2\nonumber\\
&\quad \quad \quad\ll \left|\sum\limits_{1\leq k^\prime\leq K/m} \frac{\mu(k^\prime)}{k^\prime}\right|^2 + m^2 K^{-1/2}. \label{eq1}
\end{align*}
The right-hand side of above formula tends to $0$ as $K\rightarrow \infty$. Hence
\[
z^m \in \overline{\text{Span}}^{\|\cdot\|_{H^2}}\{F_k:k=2,3,\ldots\}
\]
for all $m\geq 1$. Now Lemma \ref{lem_equiv_H^2_A_1^2} shows that
\[
m z^{m-1} \in \overline{\text{Span}}^{\|\cdot\|_2}\{f_k:k=2,3,\ldots\}
\]
for $m=1,2,\ldots$. The proof is completed.
\end{proof}

Next, we prove Theorem \ref{thm_notin_Hbot}.

\begin{proof} [Proof of Theorem \ref{thm_notin_Hbot}]
Write
\[
\cB=\left\{f\in \cH: \frac{f(z)}{1-z}\in \cH\right\}.
\]
Suppose that $g\in \cB\cap \cH^\bot$. Our aim to is show that $g=0$.

Since $\frac{g(z)}{1-z}\in \cH=A_1^2(\bD)$, then $\frac{g(z)}{1-\overline{z}}\in L^2(\bD, dA_1)$. We use $\<\cdot,\cdot\>$ to denote the inner product on either $L^2(\bD,dA_1)$ or $\cH$. Now
\[
\left\<f_k(z),\frac{g(z)}{1-\overline{ z}}\right\>=\left\<(1-z)s_k(z),\frac{g(z)}{1-\overline{ z}}\right\>= \<s_k(z), g(z)\> = 0.
\]
By Theorem \ref{thm_spanfk=A_1^2}, we have ${\overline{\text{Span}}}^{\|\cdot\|_2}\{f_k(z)\}=\cH$. It follows that $\left\<f,\frac{g(z)}{1-\overline{z}}\right\>=0$ for any $f\in \cH$. In particular,
\[
\left\<z^v,\frac{g(z)}{1-\overline{z}}\right\>=0,\quad (v\in \bN_0).
\]

Let $g(z)=\sum\limits_{n=0}^\infty a_n z^n$. Then
\begin{align*}
\frac{g(z)}{1-\overline{z}} &= \left(\sum\limits_{n=0}^\infty a_n z^n\right)\left(\sum\limits_{n=0}^\infty (\overline{  z})^n\right) \\
&= \sum\limits_{n=0}^\infty \sum\limits_{n_1-n_2=n} a_{n_1}  |z|^{2n_2} z^n + \sum\limits_{n=1}^\infty \sum\limits_{n_2-n_1=n} a_{n_1}  |z|^{2n_1} \overline{z}^n.
\end{align*}
Note that $\<z^v, |z|^r z^s\> =0$ whenever $v,r,s,\geq 0$ and $s\neq v$. Moreover, when $v,r\geq 0$ and $s\geq 1$, we also have $\<z^v, |z|^r \overline{z}^s\> =0$. Hence
\[
0=\left\<z^v,\frac{g(z)}{1-\overline{  z}}\right\> = \sum\limits_{n_1-n_2=v} a_{n_1} \left\<z^v,|z|^{2n_2}z^v\right\>.
\]
Moreover,
\[
\left\<z^v,|z|^{2n_2}z^v\right\> = 2\int_0^1 r^{2v+2n_2+1}(1-r^2)dr = \frac{1}{(v+n_2+1)(v+n_2+2)}.
\]
We have that
\[
0= \sum\limits_{n_1-n_2=v}  \frac{a_{n_1}}{(v+n_2+1)(v+n_2+2)} = \sum\limits_{m=v}^\infty \frac{a_m}{(m+1)(m+2)},\quad (v\in \bN_0).
\]
By induction, one can see that $a_m=0$ for all $m\in \bN_0$, i.e., $g=0$. So $\cH^\bot \cap \cB=\{0\}$. The proof is completed.
\end{proof}

%\begin{lemma}
%Éè$\|f_n-f\|_{A_1^2}\rightarrow 0$ $(n\rightarrow \infty)$£¬ÇÒ$\frac{f_n}{1-z},\, \frac{f}{1-z}\in A_1^2$. Èô´æÔÚ$\delta>0$£¬Ê¹µÃ$f_n$ ÔÚ$\{z\in \bD: |1-z|<\delta\}$ÄÚÒ»ÖÂÊÕÁ²ÓÚ$f$£¬Ôò$\left\|\frac{f_n}{1-z}-\frac{f}{1-z}\right\|_{A_1^2}\rightarrow 0$.
%\end{lemma}

\section{Operators and Operator Algebras}
\label{sec_3}

The space $\cH=A_1^2(\bD)$ has an orthogonal basis $\{z^m:\, m\in \bN_0\}$. Recalling \eqref{eq_Tk_def}, for $k\in \bN$, we have
\[
T_k (z^m) = k^{1/2}z^{km+k-1}\frac{1-z^k}{1-z}= k^{1/2}\sum\limits_{l=0}^{k-1} z^{km+k-1+l}, \quad (m\in \bN_0, \, z\in \bD).
\]

%That is to say, for $f(z)=\sum\limits_{n=0}^\infty a_n(f)z^n$, we have
%\[
%(T_k f)(z) = k^{1/2}\sum\limits_{m=0}^\infty a_m(f) z^{km+k-1}\sum\limits_{l=0}^{k-1} z^{l}  = %\sum\limits_{n= k-1}^\infty a_{\kappa(n)}(f)z^n.
%\]
The operator $T_1=I$ is the identity operator on $\cH$. For $j,k\in \bN$, we have
\begin{align*}
(T_jT_k) (z^m) &= k^{1/2} \sum\limits_{l=0}^{k-1} T_j(z^{km+k-1+l}) = (jk)^{1/2}  \sum\limits_{l=0}^{k-1}\sum\limits_{n=0}^{j-1}z^{j(km+k-1+l)+j-1+n} \\
&= (jk)^{1/2}\sum\limits_{l=0}^{k-1}\sum\limits_{n=0}^{j-1}z^{jkm+jk-1+(jl+n)} = T_{jk}(z^m), \quad (m\in \bN_0, \, z\in \bD).
\end{align*}
So $T_jT_k=T_{jk}$ for any $j,k\in \bN$. Let $f(z)=\sum\nolimits_{n=0}^\infty a_nz^n$ be a function in $\cH$. % satisfying $\|f\|=\sum\nolimits_{n=0}^\infty \frac{2|a_n|^2}{(n+1)(n+2)}=1$.
Then
\[
(T_k f)(z) = k^{1/2}\sum\limits_{n=0}^\infty a_n \sum\limits_{l=0}^{k-1}z^{kn+k-1+l} = k^{1/2}\sum\limits_{m=k-1}^\infty a_{\kappa(k;m)}z^m,
\]
where we define
\begin{equation} \label{eq_def_kappa}
\kappa(k;m)=\lfloor (m+1)/k\rfloor-1.
\end{equation}
So
\begin{align*}
\|T_kf\|^2_2 &= k\sum\limits_{m=k-1}^\infty \frac{2|a_{\kappa(k;m)}|^2}{(m+1)(m+2)} = k\sum\limits_{n=0}^\infty 2|a_n|^2 \sum\limits_{l=kn+k-1}^{kn+2k-2}\frac{1}{(l+1)(l+2)}\\
&= k\sum\limits_{n=0}^\infty 2|a_n|^2 \cdot \left(\frac{1}{kn+k}-\frac{1}{kn+2k}\right) = \sum\limits_{n=0}^\infty \frac{2|a_n|^2}{(n+1)(n+2)}  = \|f\|^2_2.
\end{align*}
It follows that $\|T_k\|=1$. Note that, for $m,n\in\bN_0$,
\begin{align*}
\<z^m,T_k^\ast z^n\>&=\<T_k z^m, z^n\> = k^{1/2}\sum\limits_{l=0}^{k-1}\<z^{km+k-1+l},z^n\> \\
&=
\begin{cases}
\frac{2k^{1/2}}{(n+1)(n+2)},\quad &\text{if }km+k-1\leq n\leq km+2k-2,\\
0,\quad &\text{otherwise}.
\end{cases}
\end{align*}
We conclude that the adjoint operator $T_k^\ast$ is given by
\begin{equation} \label{eq_T_k_ast_z_n}
T_k^\ast z^n =
\begin{cases}
\frac{k^{1/2}(\kappa(k;n)+1)(\kappa(k;n)+2)}{(n+1)(n+2)}z^{\kappa (k;n)},\quad &\text{if }n\geq k-1,\\
0,\quad &\text{if }n\leq k-2.
\end{cases}
\end{equation}
Suppose that $f(z)=\sum\limits_{n=0}^\infty a_n z^n$. Then
\begin{align}
(T_k^\ast f)(z) &=  \sum\limits_{n=k-1}^\infty \frac{k^{1/2}(\kappa(k;n)+1)(\kappa(k;n)+2)}{(n+1)(n+2)} a_n z^{\kappa(k;n)} \nonumber\\
&= k^{1/2}\sum\limits_{m=0}^\infty \left(\sum\limits_{n=km+k-1}^{km+2k-2}\frac{a_n }{(n+1)(n+2)} \right)(m+1)(m+2)z^m. \label{eq_T_k_ast_f}
\end{align}

Note that
\begin{align*}
T_k^\ast T_k z^m &=k^{1/2}\sum\limits_{l=0}^{k-1} T_k^\ast \left(z^{km+k-1+l}\right) \\
&= k \sum\limits_{l=0}^{k-1}\frac{(m+1)(m+2)}{(km+k+l)(km+k+l+1)} z^m = z^m,\quad (m\in \bN_0).
\end{align*}
Hence $T_k^\ast T_k = I$. It follows that $(T_k T_k^\ast)^2 = T_kT_k^\ast = (T_kT_k^\ast)^\ast$. In particular, the operators $E_k=T_k T_k^\ast$ $(k\in \bN)$ is also a projection. %Let $\cE_k$ be the closure of $\text{Ran}E_k$ in $\cH$.

Now we prove Lemma \ref{lem_Tk_sm}.

\begin{proof} [Proof of Lemma \ref{lem_Tk_sm}]
Note that
\[
s_m(z) = \sum\limits_{n=0}^\infty \left\{\frac{n+1}{m}\right\}z^n = \frac{1}{m}\sum\limits_{c=0}^{m-2} (c+1) \sum\limits_{h=0}^\infty  z^{mh+c}.
\]
So
\begin{align*}
k^{-1/2}(T_k s_m)(z) = \frac{1}{m}\sum\limits_{c=0}^{m-2} (c+1) \sum\limits_{h=0}^\infty  \sum\limits_{l=0}^{k-1} z^{k(mh+c)+k-1+l}.
\end{align*}
Note that
\begin{align*}
(s_{km}-\frac{1}{m}s_k)(z) =  \sum\limits_{n=0}^\infty \left(\left\{\frac{n+1}{mk}\right\}-\frac{1}{m}\left\{\frac{n+1}{k}\right\}\right)z^n.
\end{align*}
The terms with $0\leq n\leq k-2$ equal zero. For $n\geq k-1$, we let $n=mkh+(c+1)k+l-1$ for some $h\in \bN_0$, $0\leq c\leq m-1$ and $0\leq l\leq k-1$. When $c=m-1$, we have
\[
\left\{\frac{n+1}{mk}\right\}-\frac{1}{m}\left\{\frac{n+1}{k}\right\} =
\frac{l}{mk}-\frac{1}{m}\cdot\frac{l}{k} = 0.
\]
When $0\leq c\leq m-2$, we obtain
\[
\left\{\frac{n+1}{mk}\right\}-\frac{1}{m}\left\{\frac{n+1}{k}\right\} =
\frac{(c+1)k+l}{mk}-\frac{1}{m}\cdot\frac{l}{k} = \frac{c+1}{m}.
\]
It follows that $(s_{km}-\frac{1}{m}s_k)(z)=k^{-1/2}(T_k s_m)(z)$. The lemma follows.
\end{proof}

%\begin{align*}
%(T_k^\ast s_l)(z) &= k^{1/2}\sum\limits_{m=0}^\infty \left(\sum\limits_{n=km+k-1}^{km+2k-2}\frac{\{(n+1)/l\}}{(n+1)(n+2)} \right)(m+1)(m+2)z^m \\
%\end{align*}

Next, we prove Theorem \ref{thm_TN_generate_BH}

\begin{proof}
Let $\fR$ be the von-Neumann algebra generated by $L_\bN$ in $B(\cH)$. Denote by $\fR^\prime$ the commutant of $\fR$. That is to say, for $S\in \fR^\prime$ and $T\in \fR$, one always has $ST=TS$. Note that $\fR=B(\cH)$ is equivalent to $\fR^\prime =\bC I$. Let $S\in \fR^\prime$. Suppose that $Sz^l =\sum\nolimits_{n=0}^\infty a_{l,n} z^n$ $(l\in \bN_0)$, where $a_{l,n}\in \bC$. It is sufficient to prove that $S=a_{0,0}I$.

First, we shall verify that $a_{l,N}=0$ whenever $N>l$. Note that $T_k^\ast z^l=0$ for all $k\geq l+2$. Then $(T_k^\ast S)z^l = (ST_k^\ast)z^l = S 0 =0$. Since
\begin{equation} \label{eq_Tkast_S}
T_k^\ast (S z^l) = k^{1/2}\sum\limits_{m=0}^\infty \left(\sum\limits_{n=km+k-1}^{km+2k-2}\frac{a_{l,n} }{(n+1)(n+2)} \right)(m+1)(m+2)z^m,
\end{equation}
we obtain that
\begin{equation} \label{eq_=0_induction_0}
\sum\limits_{n=km+k-1}^{km+2k-2}\frac{a_{l,n} }{(n+1)(n+2)} = 0,\quad (k\geq l+2, \, m\in \bN_0).
\end{equation}
In view of the fact $\gcd(N,N+1)=1$, we can taking an element $M$ from the non-empty set
\[
\{(N+1)m+2N:\,m\in \bN_0\} \cap \{(N+2)m+2N+2:\, m\in \bN_0\}.
\]
It follows from \eqref{eq_=0_induction_0}, by summing with $k=N+1$ over $0\leq m\leq (M-2N)/(N+1)$, and by summing with $k=N+2$ over $0\leq m\leq (M-2N-2)/(N+2)$, that
\[
\sum\limits_{n=N}^{M}\frac{a_{l,n} }{(n+1)(n+2)} = \sum\limits_{n=N+1}^{M}\frac{a_{l,n} }{(n+1)(n+2)}=0.
\]
Thus $a_{l,N}=0$ for $N>l$. In particular, one concludes that $S1 = a_{0,0}$.

Now we assume as inductive hypothesis that $Sz^l = a_{0,0}z^l$ for all $l\leq L-1$, where $L\geq 1$. Our aim is to show that $Sz^L = a_{0,0}z^L$. Recall that, for $k\geq 2$,
\[
T_k^\ast z^L =
\begin{cases}
\frac{k^{1/2}(\kappa(k;L)+1)(\kappa(k;L)+2)}{(L+1)(L+2)}z^{\kappa (k;L)},\quad &\text{if }L\geq k-1,\\
0,\quad &\text{if }L\leq k-2.
\end{cases}
\]
In view of $\kappa(k;L)<L$, one deduces by inductive hypothesis that
\[
(T_k^\ast S)z^L =  S(T_k^\ast z^L) =
\begin{cases}
\frac{k^{1/2}(\kappa(k;L)+1)(\kappa(k;L)+2)}{(L+1)(L+2)}a_{0,0} z^{\kappa (k;L)},\quad &\text{if }k\leq L+1,\\
0,\quad &\text{if }k\geq L+2.
\end{cases}
\]
%On the other hand, we have
%\[
%T_k^\ast (Sz^L) = k^{1/2}\sum\limits_{m=0}^\infty \left(\sum\limits_{n=km+k-1}^{km+2k-2}\frac{a_{L,n} }{(n+1)(n+2)} \right)(m+1)(m+2)z^m,
%\]
For $k\leq L+1$, combining \eqref{eq_Tkast_S} one obtains that
\begin{equation} \label{eq_=0_induction}
\sum\limits_{n=km+k-1}^{km+2k-2}\frac{a_{L,n} }{(n+1)(n+2)} =
\begin{cases}
\frac{a_{0,0}}{(L+1)(L+2)}, \quad &\text{if }m= \kappa(k;L),\\
0,\quad &\text{otherwise}.
\end{cases}
\end{equation}
Now taking $k=L+1$, we have $\kappa(k;L)=0$. So
\[
\frac{a_{0,0}}{(L+1)(L+2)} = \sum\limits_{n=L}^{2L}\frac{a_{L,n} }{(n+1)(n+2)} = \frac{a_{L,L}}{(L+1)(L+2)},
\]
which leads to $a_{L,L}=a_{0,0}$.

It remains to prove that $a_{L,0}=\ldots = a_{L,L-1}=0$ under the inductive hypothesis $Sz^l = a_{0,0}z^l$ for $l\leq L-1$.  When $L+2$ is not a prime, write $L+2 =rt$ for some $r,t\geq 2$. Note that
\begin{align*}
(ST_r)z^{t-2} &= r^{1/2}S(z^{rt-r-1}+\ldots + z^{L}) \\
&= r^{1/2}a_{0,0}(z^{rt-r-1}+\ldots + z^{L-1})+r^{1/2}(a_{L,0}+a_{L,1}z+\ldots+ a_{L,L-1}z^{L-1}+a_{0,0}z^L).
\end{align*}
In view of $t-2<L$, we also have
\[
(T_r S)z^{t-2} = a_{0,0}T_r z^{t-2} =  r^{1/2}a_{0,0}(z^{rt-r-1}+\ldots + z^{L}).
\]
It follows from $ST_r=T_rS$ that $a_{L,0}=\ldots = a_{L,L-1}=0$. When $L+2$ is a prime, the number $L$ is odd. Note that
\begin{align*}
(T_2^\ast T_{L+2}) 1&= (L+2)^{1/2} T_2^\ast (z^{L+1}+\ldots + z^{2L+2}) \\
&=
2^{1/2}(L+2)^{1/2}\cdot \left(\frac{L+1}{4(L+2)}z^{(L-1)/2}+\frac{1}{2}(z^{(L+1)/2}+\ldots + z^L)\right)
\end{align*}
Hence
\begin{align*}
&2^{-1/2}(L+2)^{-1/2}\cdot (S T_2^\ast T_{L+2})1 = \frac{a_{0,0}(L+1)}{4(L+2)}z^{(L-1)/2}\\
&\quad\quad\quad +\frac{a_{0,0}}{2}(z^{(L+1)/2}+\ldots +z^{L-1})+\frac{1}{2}(a_{L,0}+\ldots a_{L,L-1}z^{L-1}+a_{0,0}z^L).
\end{align*}
And
\begin{align*}
&2^{-1/2}(L+2)^{-1/2}\cdot (T_2^\ast T_{L+2}S)1 = 2^{-1/2}(L+2)^{-1/2}\cdot a_{0,0}(T_2^\ast T_{L+2})1 \\
&\quad\quad\quad =\frac{a_{0,0}(L+1)}{4(L+2)}z^{(L-1)/2}+\frac{a_{0,0}}{2}(z^{(L+1)/2}+\ldots + z^L).
\end{align*}
It follows from $ST_2^\ast T_{L+2}=T^\ast_2 T_{L+2}S$ that $a_{L,0}=\ldots = a_{L,L-1}=0$.
%. When $L$ is even, we have
%\begin{align*}
%(S T_2^\ast T_{L+2})1 &= (L+2)^{1/2}(ST_2^\ast)(z^{L+1}+\ldots +z^{2L+2})= 2^{1/2}(L+2)^{1/2}\cdot \frac{1}{2} S(z^{L/2}+\ldots + z^{L})\\
%&= 2^{1/2}(L+2)^{1/2}\cdot \frac{1}{2}\left(a_{0,0}(z^{L/2}+\ldots +z^{L-1})+(a_{L,0}+\ldots a_{L,L-1}z^{L-1}+a_{0,0}z^L)\right).
%\end{align*}
%And
%\[
%(T_2^\ast T_{L+2}S)1 = a_{0,0}(T_2^\ast T_{L+2})1 =2^{1/2}(L+2)^{1/2}\cdot \frac{1}{2} a_{0,0}(z^{L/2}+\ldots + z^{L}).
%\]
%Then $a_{L,0}=\ldots = a_{L,L-1}=0$. When $L$ is odd, we have
%\begin{align*}
%(S T_2^\ast T_{L+2})1 &= (L+2)^{1/2}(ST_2^\ast)(z^{L+1}+\ldots +z^{2L+2})= 2^{1/2}(L+2)^{1/2}\cdot \frac{1}{2} S(z^{L/2}+\ldots + z^{L})\\
%&= 2^{1/2}(L+2)^{1/2}\cdot \frac{1}{2}\left(a_{0,0}(z^{L/2}+\ldots +z^{L-1})+(a_{L,0}+\ldots a_{L,L-1}z^{L-1}+a_{0,0}z^L)\right).
%\end{align*}
By induction, we have shown that $\fR^\prime =\bC I$. Then $\fR=B(\cH)$ and the theorem follows.
\end{proof}

\begin{proof} [Proof of Corollary \ref{cor_RH_Lkast_invariant}]
Suppose that Riemann Hypothesis is true. Then $\cH_0=\cH$ by Theorem 0. It follows that $\cH_0$ is invariant under $T_k^\ast$ for all $k\in \bN$. Conversely, suppose that $\cH_0$ is $T_k^\ast$-invariant for all $k\in \bN$. By Lemma \ref{lem_Tk_sm}, the subspace $\cH_0$ is also invariant under $T_k$ $(k\in \bN)$. Then $\cH_0$ is invariant under any operator in the von-Neumann algebra generated by $T_\bN$. Combining Theorem \ref{thm_TN_generate_BH}, we conclude that $\cH_0=\cH$. Then Riemann hypothesis follows from Theorem 0.
\end{proof}

%Question: Let $p\in \cP$. Is the von Neumann algebra generated by $\{L_k:\,k\in \bN\}$ and $\{L _k^\ast:\, k\in \cP\setminus \{p\}\}$ the whole $B(\cH)$?

It follows from \eqref{eq_T_k_ast_z_n} that, for any given $k\in \bN$ and $m\in \bN_0$, the Hilbert subspace spanned by $\{(T_k^\ast)^l z^m: \, l\in \bN_0\}$ has finite dimension. Furthermore, for any given $n\in \bN_0$,
\[
\text{Span}\{z^m:\, m=0,1,2,\ldots,n\}
\]
is a common invariant subspace of $T_k^\ast$ $(k\in \bN)$.

Moreover, the function $\beta(z)=\frac{1}{1-z}$ has Taylor expansion $\beta(z)=\sum\limits_{n=0}^\infty z^n$. And $\|\beta\|_2=2$. By \eqref{eq_T_k_ast_f}, we have
\begin{align*}
(T_k^\ast \beta)(z) &= k^{1/2}\sum\limits_{m=0}^\infty \left(\sum\limits_{n=km+k-1}^{km+2k-2}\frac{1}{(n+1)(n+2)} \right)(m+1)(m+2)z^m \\
&= k^{-1/2} \sum\limits_{m=0}^\infty z^m = k^{-1/2}\beta(z).
\end{align*}
So $\beta$ is a common eigenfunction of $T_k^\ast$ for all $k\in \bN$.

Finally, let us suppose that $f(z)=\sum\nolimits_{n=l}^\infty a_n z^n$ with $l\geq 0$ and $a_l\neq 0$. Consider $k\in \bN$ with $k\geq 2$. It is immediate that $T_k f\neq 0$. Since $(T_k f)(z) = \sum\nolimits_{n=kl+k-1}^\infty b_n z^n$ for some coefficients $b_n$, one sees that $(T_k f)(z)\neq \lambda f(z)$ for all $\lambda\neq 0$. Therefore, the operators $T_k$ $(k\geq 2)$ do not have common eigenfunctions.

%Try frame for $\{s_k\}$

%or framing for the Banach algebra generated by $\{L_k\}$

%dilations?

%non-commutative measures? quantum measures?

\bigskip

\textbf{Acknowledgements.} Part of the research began when the author was a postdoc in Academy of Mathematics and Systems Science, Chinese Academy of Sciences. The author is deeply grateful to his supervisor Professor Liming Ge. The author would also like to thank Fei Wei and Jingyang Li for helpful discussions. This work is supported by National Natural Science Foundation of China (Grant No. 11701549).

%%%%%%%%%%%%%%%%%%%%%%%%%%%%%%%%%%%%%%%%%%%%%%%%%%%%%%%%%
%%%%%%%%%%%%%%%%%%%% Bibliography %%%%%%%%%%%%%%%%%%%%%%%
%%%%%%%%%%%%%%%%%%%%%%%%%%%%%%%%%%%%%%%%%%%%%%%%%%%%%%%%%

%\newpage
%\ \thispagestyle{empty}
%%%%%%%%%%%%%%%%%%%%%%%%%%%%%%%%%%%%%%%%%%%%%%%%%%%%%%%%%%%%%%%%%%%%%%%%%%%%%%%%%%%%%%%%
%%%%%%%%%%%%%%%%%%%%%%%%%%%%%%%%%%%%%%%%%%%%%%%%%%%%%%%%%%%%%%%%%%%%%%%%%%%%%%%%%%%%%%%%
%%%%%%%%%%%%%%%%%%%%%%%%%%%%% Publications %%%%%%%%%%%%%%%%%%%%%%%%%%%%%%%%%%%%%%%%%%%%%
%%%%%%%%%%%%%%%%%%%%%%%%%%%%%%%%%%%%%%%%%%%%%%%%%%%%%%%%%%%%%%%%%%%%%%%%%%%%%%%%%%%%%%%%
%%%%%%%%%%%%%%%%%%%%%%%%%%%%%%%%%%%%%%%%%%%%%%%%%%%%%%%%%%%%%%%%%%%%%%%%%%%%%%%%%%%%%%%%

\end{document}